\documentclass[12pt,oneside]{amsart}
\usepackage[inner=23mm,outer=23mm,tmargin=27mm,bmargin=27mm]{geometry}
\usepackage[utf8]{inputenc}
\usepackage[T1]{fontenc}
\usepackage{amssymb}
\usepackage{geometry}
\usepackage{amsmath}
\usepackage{amsthm}
\usepackage{graphicx}
\usepackage{color,xcolor}
\usepackage[normalem]{ulem}
\usepackage[english]{babel}
\usepackage[onehalfspacing]{setspace}
\usepackage{verbatim}
\usepackage[hidelinks]{hyperref}
\usepackage{mathtools}

\newcommand{\R}{\ensuremath{\mathbb{R}}}
\newcommand{\N}{\ensuremath{\mathbb{N}}}

\newcommand{\PP}{\ensuremath{\mathbb{P}}}

\newcommand{\E}{\ensuremath{\mathbb{E}}}

\newcommand{\dint}{{\rm d}}

\newcommand{\norm}[1]{\left\Vert#1\right\Vert}

\renewcommand{\rho}{\varrho}

\DeclareMathOperator{\APP}{APP}

\newcommand{\set}[1]{\left\{#1\right\}}
\newcommand{\abs}[1]{\left|#1\right|}
\newcommand{\brackets}[1]{\left(#1\right)}

\renewcommand{\phi}{{\varphi}}

\DeclareMathOperator{\dist}{dist}

\DeclareMathOperator{\argmin}{argmin}

\newtheorem{thm}{Theorem}

\theoremstyle{plain}
\newtheorem{lemma}[thm]{Lemma}

\newtheorem{prop}[thm]{Proposition}
\theoremstyle{definition}
\newtheorem{rem}[thm]{Remark}

\title[On the power of random information]{On the power of random information} 

\author{Aicke Hinrichs 
\and David Krieg 
\and Erich Novak 
\and Joscha Prochno 
\and Mario Ullrich}

\address[A.~Hinrichs, M.~Ullrich]{Institut f\"ur Analysis,
Johannes Kepler Universit\"at Linz,
Altenbergerstrasse 69, 4040 Linz, Austria}

\address[D.~Krieg, E.~Novak]{Mathematisches Institut, Universit\"at Jena, 
Ernst-Abbe-Platz~2, 07743 Jena, Germany}

\address[J.~Prochno]{Institut f\"ur Mathematik \& Wissenschaftliches Rechnen, 
Karl-Franzens-Universit\"at Graz, 
Heinrichstrasse 36, 8010 Graz, Austria}

\date{\today}

\begin{document}

\begin{abstract}
We study approximation and integration problems
and compare the quality  of optimal information
with the quality of random information.
For some problems random information
is almost optimal and for some other problems random
information is much worse than optimal information.
We prove new results and give a short survey 
of known results.
\end{abstract}

\maketitle
%   \tableofcontents

\section{Introduction} 

% A central objective in complexity theory is to find
% optimal algorithms using optimal information
% about the problem instance.

Optimal algorithms use optimal information 
about the problem instances.
Quite often, we do not have access to optimal information. 
One reason may be that 
we do not know how to choose the most significant measurements.
%E  It may also happen that we do not 
%E  even get to choose the parameters.
In this paper, we assume that information
comes in randomly, which is a standard assumption
in learning theory and uncertainty quantification.
We ask the following question:
\begin{center}
What is the typical quality 
of random information?
\end{center}
Of course, random information cannot be better than optimal information.
In many cases, however, it will turn out that
random information is practically as good as optimal information while
in other situations random information is a little worse.
Sometimes, random information is completely useless.

To phrase our question precisely,
we need to clarify how we measure the
quality of information 
and what we mean by random.
The first is done with the so-called
radius of information,
which is the worst case error of the best algorithm
that uses nothing but the given information
and the a priori knowledge about the
problem instance.
Random information, on the other hand, 
shall be obtained from
a certain number of independent measurements
that all follow the same law.
% Of course, there is no right or wrong choice
% for the underlying distribution. 
We study the question for two different natural choices of 
uniform distributions.

Let us go a little more into detail.
A linear problem is given by
a linear solution operator $S$ that maps from a convex and symmetric subset $F$ 
of a normed space to a normed space $G$
and a class $\Lambda$ of continuous linear functionals on $F$,
the class of admissible measurements.
We may think of an integration problem,
where $S(f)$ is the integral of a function $f$,
or a recovery problem,
where $S$ is an embedding.
One wants to approximate the solution $S(f)$
for unknown $f\in F$
based on $n$ of these measurements
such that we can guarantee a small error with
respect to the norm in $G$.
We refer the reader to
\cite{NW08,TWW88}
for a detailed exposition.
An information mapping has the form
\begin{equation}
\label{eq:information_mapping}
 N_n:F\to \R^n,
 \quad
 N_n(f)=\big(L_1(f),\hdots,L_n(f)\big),
\end{equation}
where $L_1,\dots,L_n\in\Lambda$.
The power or quality of the information mapping is measured by the
radius of information.
This is the worst case error of the best algorithm
$A_n=\phi\circ N_n$ based on this information, i.e.,
\begin{equation}
 r(N_n) = 
 \inf_{\phi: \R^n\to G} \sup_{f\in F} \norm{S(f)-\phi(N_n(f))}_G.
\end{equation}
We say that information is optimal, and write $N_n^*$ instead of $N_n$, if
\begin{equation}
 r(N_n^*) = \inf_{N_n} r(N_n).
\end{equation}
Here we study random information
of the form \eqref{eq:information_mapping},
where the random functionals $L_1,\dots,L_n\in\Lambda$
are independent and identically distributed.
The goal is to compare 
$$
 r(N_n^*)
 \qquad\text{vs.}\qquad
 \E [r(N_n)],
$$
the radius of optimal information
and the expected radius of random information.

If the infimum and the expected value are comparable,
this means that there are many good algorithms
based on many different information mappings.
In this case, 
optimal information
and therefore optimal algorithms are
not very special.
On the other hand, if the infimum is significantly smaller
than the expected value,
this means that optimal information is indeed special.
It seems to be an interesting characteristic
of a problem whether 
optimal information is special or not. Of course, the results of our comparisons
heavily depend on the distribution
underlying the measurements.
While the question may be interesting for many distributions,
we feel that there often is a natural choice.
If the distribution only depends on the class $\Lambda$
of admissible measurements, 
collecting random information 
might even be a good idea if optimal information is available.
It may happen that we do not loose much in
terms of the radius but gain the following useful properties:
\begin{itemize}
 \item Since the distribution is independent of $n$, 
 it is easy to increase the number of
 measurements if our current approximation is not yet
 satisfactory.
 \item The information can be used for
 many different input classes $F$,
 solution operators $S$, and target spaces $G$.
 In that sense it is \textit{universal}.
\end{itemize}
%\vspace*{-4mm}
The second property
does not mean that the corresponding algorithm $A_n=\phi\circ N_n$
is universal. The optimal choice of $\phi$ 
may depend on $F$, $S$, and $G$.
It is well known, see again \cite{TWW88}, that any 
interpolatory algorithm is optimal up to a factor two. 

We note that examples of the sort
`\textit{random information is good}' 
can be deduced from work using the probabilistic method 
to prove upper bounds of $r(N_n^*)$. 
Here one uses upper bounds of 
$\E [r(N_n)]$ and the trivial inequality 
$r(N_n^*)  \le \E [r(N_n)]$, 
see
\cite{HNWW01,NW08,NW10,NW12,SW98}
for many such results.
The idea is to introduce a random
family of algorithms
and to show that the expected worst case error
is small.
This is used to obtain the existence of good algorithms.
However, it actually implies that \textit{most}
of the algorithms in that family are good.
Therefore, the expected radius of the 
random information underlying these algorithms must be small as well.
We note that for this approach only upper bounds of 
$\E [r(N_n)]$ are needed and lower bounds are usually not studied. 

\section{Uniformly distributed function evaluations} \label{s2}

If the input class $F$ consists of functions, one natural restriction of the class of information functionals is so called \textit{standard information} consisting of function evaluations. Moreover, if the domain of definition of functions in $F$ carries a probability measure, a natural source of random information are function values $f(x)$ with $x$ distributed according to the probability measure. In this section, we analyze this setting for different input classes  $F$ of smooth functions defined on the unit cube $[0,1]^d$. Then the random information consists of function values $f(x)$ with $x$ uniformly distributed in $[0,1]^d$.

\subsection{Integration and Approximation in univariate Sobolev spaces of smoothness 1} 
\label{s2.1}

We consider the $L_q$-approximation problem 
$$
\APP  \colon  W_p^1 ([0,1]) \to L_q ([0,1])
$$
given by $\APP(f)=f$ using function values as information. 
% Such information is often called \textit{standard information}. 
The Sobolev spaces $W_p^1 ([0,1])$ are usually normed by 
\[
\|f\|_{W^1_p([0,1])} \,:=\,  \left(\|f\|_p^p + \|f'\|_p^p\right)^{1/p}.
\] 
It is well known that the optimal information is given by
$$
N_n^\ast (f) = \big( f(1/n), f(3/n), \dots, f((2n-1)/n) \big),  
$$
i.e., point evaluations at equidistant points. 
The minimal radius of information then satisfies
$$
r(N_n^\ast) \asymp n^{-1 + (1/p-1/q)_+} ,
$$
see \eqref{triebel}. 

Here and later we use asymptotic notation as follows. For sequences $(a_n)$ 
and $(b_n)$ of positive real numbers, we write $a_n \preccurlyeq b_n$ and 
$b_n \succcurlyeq a_n$ to indicate that there exists a constant 
$C\in(0,\infty)$ such that $a_n \leq C\, b_n$ for all $n$. 
We shall write $a_n \asymp b_n$ if $a_n \preccurlyeq b_n$ and 
$b_n \preccurlyeq a_n$. Sometimes we use similar notation for double sequences.

We will show that information given by $n$ independent and uniformly distributed 
points in $[0,1]$ is as good as optimal information provided that $p>q$, whereas
in the case $p\le q$ there is a loss of a logarithmic factor.

\begin{thm} \label{thm:appd1s1}
Consider $L_q$-approximation of functions from $W^s_p([0,1])$ 
with $1\le p,q \le \infty$, 
using function values at uniformly distributed points in $[0,1]$.
%Let $1\le p,q \le \infty$. 
 \[
  \E [r(N_n)] \asymp
 \begin{cases}
 r(N_n^\ast) \asymp n^{-1} & \text{if} \quad p>q \\
 & \\
 r(N_{n/\log n}^\ast) \asymp \left( \frac{n}{\log n} 
  \right)^{-1 + 1/p-1/q} & \text{if} \quad p \le q.
 \end{cases}
 \]
% \textcolor{red}{for $p>q$,}
% $$ 
%  \EE (r(N_n)) \asymp r(N_n^\ast) \asymp n^{-1} 
% $$
% and \textcolor{red}{for $p \le q$,}
% $$ 
% \EE (r(N_n)) \asymp r(N_{n/\log n}^\ast) \asymp \left( \frac{n}{\log n} 
% \right)^{-1 + 1/p-1/q} 
% $$ 
\end{thm}

The proof of Theorem \ref{thm:appd1s1} is based on the following formula for 
the radius of information $r(N_n)$ for information $N_n(f)= \big( f(x_1), f(x_2), \dots, f(x_n) \big)$,
where we assume $0=x_0 \le x_1 \le \dots \le x_n  \le x_{n+1}=1$.

\begin{lemma} \label{lem:appd1s1}
  $$
   r(N_n) \asymp
        \left\{\begin{array}{ll}
\displaystyle \left( \sum_{i=0}^n (x_{i+1}-x_i)^{(pq+p-q)/(p-q)} \right)^{1/q-1/p} & \text{if} \quad p>q\\
       \displaystyle \max_{0 \le i \le n} (x_{i+1}-x_i)^{1-1/p+1/q} &\text{if} \quad  p\le q.\\
\end{array}\right.
$$
\end{lemma}

\begin{proof}
 Note that in this case the radius of information satisfies
\[
r(N_n)\asymp \sup_{\substack{f\in W_p^1([0,1])\\ N_n(f)=0}} 
\frac{\norm{f}_q}{\|f\|_{W_p^1}},
\]
with an equivalence constant of at most 2,
see \cite[Section~4.2]{NW08}.
 We start with the lower bound in the case $p \le q$.
 Let $i_0 \in \{0,1, \dots,n\}$ be such that 
 $$
  \ell = x_{i_0+1} - x_{i_0} = \max_{0 \le i \le n} (x_{i+1}-x_i)
 $$ 
 and consider a continuous piecewise linear function $f:[0,1] \to \R$ 
 that is zero outside of the  interval $(x_{i_0},x_{i_0+1})$, satisfies $f'(x)=1$ in the
 left half of this interval and $f'(x)=-1$ in the right half.
 Then 
  $$ 
   N_n(f)=(0,\dots,0), \quad
   \|f\|_q \asymp \ell^{1+1/q}, \quad \text{and} \quad
   \|f\|_{W_p^1} \asymp  \ell^{1/p}
  $$   
  show that $r(N_n) \succcurlyeq \ell^{1-1/p+1/q}$.

  We turn to the lower bound in the case $p > q$.
  In this case, let $f:[0,1] \to \R$ be a continuous piecewise linear function
  satisfying $f(x_i)=0$ for $i=0,1,\dots,n+1$ and $f'(x)=\ell_i^\alpha$ in the
  left half of each interval $[x_i,x_{i+1}]$ and $f'(x)=-\ell_i^\alpha$ in the right half, 
  where $\ell_i=x_{i+1}-x_i$ and $\alpha = q/(p-q)$. Then $N_n(f)=(0,\dots,0)$ and 
  $$ 
   \|f\|_q \asymp \left( \sum_{i=0}^n \ell_i^{(pq+p-q)/(p-q)} \right)^{1/q} \quad \text{and} \quad
   \|f\|_{W_p^1} \asymp  \left( \sum_{i=0}^n \ell_i^{(pq+p-q)/(p-q)} \right)^{1/p}
  $$  
  show that $r(N_n) \succcurlyeq\left( \sum_{i=0}^n \ell_i^{(pq+p-q)/(p-q)} \right)^{1/q-1/p}.$
  
  Finally, we prove the upper bound. 
  Observe that $ \| f \|_\infty \le \|f'\|_1$
  holds for any function $f \in  W_1^1 ([0,1])$ with $f(0)=0$.
  Consequently, $ \| f \|_q \le \|f'\|_p$ holds for any function $f \in  W_p^1 ([0,1])$
  with $f(0)=0$. 
  By scaling, it follows that $ \| f \|_q \le  (b-a)^{1-1/p+1/q} \|f'\|_p$ 
  holds for any function $f \in  W_p^1 ([a,b])$
  satisfying either $f(a)=0$ or $f(b)=0$.
  Let now $f \in  W_p^1 ([0,1])$ be such that $N_n(f)= \big( f(x_1), f(x_2), 
  \dots, f(x_n) \big) = (0,\dots,0)$
  and $\|f'\|_p = 1$. Let $f_i$ be the restriction of $f$ to the interval $[x_i,x_{i+1}]$. Then 
  $$ 
   \| f \|_q = \left( \sum_{i=0}^n \|f_i\|_q^q \right)^{1/q}
             \le \left( \sum_{i=0}^n (x_{i+1}-x_i)^{q-q/p+1} \|f'_i\|_p^q \right)^{1/q}.
  $$
  In the case $p \le q$, let again $\ell$ be the length of the largest of the intervals $[x_i,x_{i+1}]$.
  It follows that
  \begin{align*}
    \| f \|_q &\le  \ell^{1-1/p+1/q} \left( \sum_{i=0}^n  \|f'_i\|_p^q \right)^{1/q}
              \le \ell^{1-1/p+1/q} \left( \sum_{i=0}^n  \|f'_i\|_p^p \right)^{1/p} \\
              & =  \ell^{1-1/p+1/q} \|f'\|_p = \ell^{1-1/p+1/q}.
  \end{align*}
  This implies $r(N_n) \le  \ell^{1-1/p+1/q}$ as claimed.
  In the case $p > q$, let again $\ell_i=x_{i+1}-x_i$ and $s > 0$ be given by $\frac1s = \frac1q - \frac1p$. Using H\"older's inequality, we arrive at
  \begin{align*}
    \| f \|_q &\le \left( \sum_{i=0}^n \ell_i^{s+1}  \right)^{1/s} \left( \sum_{i=0}^n  \|f'_i\|_p^p \right)^{1/p}
               = \left( \sum_{i=0}^n \ell_i^{s+1}  \right)^{1/s} \|f'\|_p \\
              &=  \left( \sum_{i=0}^n \ell_i^{s+1}  \right)^{1/s}.
  \end{align*}  
  This implies the claimed upper bound also when $p>q$.
\end{proof}

\begin{proof}[Proof of Theorem \ref{thm:appd1s1}]
 In the case $p\le q$ the proof follows from Lemma \ref{lem:appd1s1}
 and the fact that, in expectation, 
 the largest gap among $n$ independent 
 and uniformly distributed points in $[0,1]$ is of order $\log(n)/n$.
 This is a simple consequence of Lemma~\ref{lem:coupons}.
%  via analogous reasoning as in the proof of 
%  Proposition~\ref{thm:lipinfty}.
 In the case $p>q$ the proof follows from
 $$
  \E\bigg[ \Big( \sum_{i=0}^n \ell_i^{s+1}  \Big)^{1/s}\bigg] \asymp \frac1n
 $$
 for any $s > 0$, where the $\ell_i$ are again the spacings of independent and uniformly distributed points in $[0,1]$. 
 In fact, the lower bound follows trivially since the expected radius of random information is larger than the radius of optimal information. 
 To prove the upper bound observe first that the function 
$s \mapsto \left( \sum_{i=0}^n \ell_i^{s+1}  \right)^{1/s}$ is 
 non-decreasing, since $\sum_{i=0}^n \ell_i=1$, and that Jensen's inequality implies
$$
 \E \bigg[\Big( \sum_{i=0}^n \ell_i^{s+1}  \Big)^{1/s}\bigg]
 \leq
 \bigg( \E  \Big[\sum_{i=0}^n \ell_i^{s+1}\Big]  \bigg)^{1/s} 
\qquad \text{whenever } s\ge1.
$$ 
Hence, it is enough to show that 
$$
 \E  \Big[\sum_{i=0}^n \ell_i^{s+1}\Big]
 \asymp
 n^{-s}
$$
for positive integers $s$.
It was proved in \cite{D1953} that
$$
 \E \Big[ \sum_{i=0}^n h(\ell_i)\Big]
 =
 n(n+1) \int_0^1 (1-r)^{n-1} h(r) \dint r
$$
for any integrable function $h : [0,1] \to \R$.
With $h(r)=r^{s+1}$, we obtain
$$
 \E  \Big[\sum_{i=0}^n \ell_i^{s+1}\Big]
 =
 n(n+1) B(s+2,n) =  \frac{(n+1)! (s+1)!}{(n+s+1)!}
 \asymp
 n^{-s},
$$
where $B(x,y)$ is the Beta function. This completes the proof.
\end{proof}

\begin{rem}
For the classes studied, i.e., when the smoothness is one, the radius of information $r(N_n)$ 
for the integration problem with 
$S(f) = \int_0^1 f(x) \, {\rm d} x$ is the same as for the $L_1$-approximation 
problem. Hence, for the integration problem for functions in $W_p^1 ([0,1])$, we have 
\[
  \E [r(N_n)] \asymp
 \begin{cases}
 r(N_n^\ast) \asymp n^{-1} & \text{if} \quad p>1 \\
 & \\
 r(N_{n/\log n}^\ast) \asymp \left( \frac{n}{\log n} 
  \right)^{-1} & \text{if} \quad p=1.
 \end{cases}
\]

Some cases of  Theorem \ref{thm:appd1s1} are known, 
sometimes in 
a slightly different 
language since authors had in mind discrepancy
(with optimal weights) instead of the 
integration problem. 
The star discrepany for $d=1$, i.e., 
$p=q=1$, was studied in~\cite{W15}. 
Here we have $r(N_n^*) = \frac{1}{2n}$. 
From Corollary~3.3 and Theorem~3.5 of \cite{W15} we know that
$$
\frac{c\log n}{n} \le \E [ r(N_n)] \le \frac{8 \log n}{n} 
$$
for some $c\in(0,\infty)$.
Therefore, random information is good, 
but not optimal. Also the case $p=2$ and $q=1$ is covered by \cite[Theorems 4.2 and 4.3]{W15}.
\end{rem}

\subsection{Approximation of multivariate Lipschitz functions}
\label{s2.2}

We study the problem of $L_q$-approximation of functions from the class
\begin{equation}
 \label{eq:defFdLip}
 F_d=\Big\{f:[0,1]^d\to \R \,:\,
 \abs{f(x)-f(y)}
 \leq \dist(x,y) \,\,\, \forall x,y\in[0,1]^d\Big\}.
\end{equation}
% of Lipschitz continuous functions on the $d$-torus
% with deterministic algorithms based on $n$
% pieces of standard information in the worst
% case setting.
Here, $\dist$ denotes the maximum metric on the $d$-torus, i.e.,
$$
 \dist(x,y)=\min_{k\in\mathbb{Z}^d} 
 \norm{x+k-y}_\infty
 \quad\text{for}\quad
 x,y\in[0,1]^d.
$$
% It is well known that optimal information
% is given by function values on a regular grid
% if $n=m^d$ for some $m\in \N$, e.g.,
% $$
%  N_n^*(f)=(f(x))_{x\in P_n^*},
%  \quad
%  P_n^*=\set{i/m \mid 0\leq i <m}^d.
% $$
We consider standard information of the form
$$
 N_n(f)=(f(x_1),\hdots,f(x_n)),
 \qquad
 x_1,\hdots,x_n\in[0,1]^d.
%   N_n(f)=(f(x))_{i \leq n}
%   \quad\text{with}\quad
%   P_n=\set{x_1,\hdots,x_n}\subset [0,1]^d.
$$
An optimal algorithm
based on $N_n$ works as follows:

\smallskip\noindent{\bf Algorithm.}
 For $N_n$ as above and $f\in F_d$,
 we define $A_n(f)=(f^++f^-)/2$, where
%  For $x\in[0,1]^d$, set
 \[
  f^+(x)=\min_{i \leq n} \big(f(x_i) + \dist(x,x_i)\big)
  \quad\text{and}\quad
  f^-(x)=\max_{i \leq n} \big(f(x_i) - \dist(x,x_i)\big).
 \]

Note that $f^+$ and $f^-$ are the maximal and the minimal 
 function in $F_d$ that interpolate $f$ at the nodes $x_1,\dots,x_n$.
We have the following relations that we prove for the sake of completeness, 
see e.g.~\cite{NW08,TWW88}.

\begin{lemma}\label{lem:radius0}
% For any nonadaptive information $N_n$,
% Algorithm~\ref{alg:optimal alg Lipschitz}
% satisfies
 $$
%   \err(A_n,\APP,F_d,L_q,\mathrm{wc})
r(N_n) =
 \sup_{f\in F_d} \norm{f-A_n(f)}_q
  =\sup_{\substack{f\in F_d\\ N_n(f)=0}} \norm{f}_q.
 $$
\end{lemma}

\begin{proof}
 Clearly $A_n$ is of the form $\phi\circ N_n$
 for some mapping $\phi:\R^n\to L_q$.
 The definition of the radius yields 
 $$
  \sup_{f\in F_d} \norm{f-A_n(f)}_q
  \geq r(N_n) \geq \sup_{\substack{f\in F_d\\ N_n(f)=0}} \norm{f}_q.
 $$
 On the other hand, any $f\in F_d$ satisfies the pointwise estimate
 \begin{equation}
 \label{eq:pointwise estimate}
  \abs{f - A_n(f)} \leq \frac{f^+-f^-}{2}.
 \end{equation}
 This implies 
 $$
 \sup_{\substack{f\in F_d\\ N_n(f)=0}} \norm{f}_q
  \geq \sup_{f\in F_d} \norm{f-A_n(f)}_q
 $$ 
 since the right-hand side of \eqref{eq:pointwise estimate}
 is an element of $F_d$ that vanishes on $\{x_1,\hdots,x_n\}$. 
 Altogether, we obtain the claimed identities.
\end{proof}

It is also well known that optimal information
is given by function values on a regular grid
if $n=m^d$ for some $m\in\N$.
% For instance, we may choose
% $$
%  N_n^*(f)=(f(x_1),\hdots,f(x_n))
%  \quad\text{with}\quad
%  \set{x_1,\hdots,x_n}=\{i/m \mid 0\leq i<m\}^d
% $$
The radius of optimal information $N_n^*$ satisfies
$$
 r(N_n^*) \asymp n^{-1/d}
$$
for all $1\leq q\leq\infty$.
This follows from the upper bound on the complexity
of uniform approximation as studied in \cite{Su78} 
and the lower bound on the complexity of numerical
integration as studied in \cite{Su79}.
%and the lower bound follows from the see \cite{Su78,Su79,No88}.
Using the technique of proof of the latter, 
we even obtain a precise formula
for the radius of optimal information.

\begin{prop}
\label{prop:lipoptimal}
 Let $n=m^d$ for some $m\in\N$.
 Then
 $$
  r(N_n^*)
  = \left\{\begin{array}{ll}
  		\displaystyle
        \frac{1}{2} \Big(\frac{d}{d+q}\Big)^{1/q}\,n^{-1/d}
        & \text{if} \quad 1\leq q<\infty,\vspace*{2mm}\\
        \displaystyle
        \frac{1}{2}\,n^{-1/d} &\text{if} \quad q=\infty.
        \end{array}\right.
 $$
% for $q<\infty$ and
% $$
%  \inf_{N_n} r(N_n)
%  = \frac{1}{2m}.
% $$
% The infima are attained for %by nonadaptive information on
% $$
%  N_n^*:F_d\to \R^n, \quad N_n(f)=(f(x))_{x\in P_n^*}.
% $$
%  $P_n=\{i/m \mid 0\leq i<m\}^d$.
\end{prop}

\begin{proof}
% Recall that
% $$
%  r(N_n)^q=
%  \sup_{\substack{f\in F_d\\ N_n(f)=0}}\,\int_{[0,1]^d} 
%  \abs{f(x)}^q~\dx.
% $$
 Let $N_n$ be as above and let $P_n=\{x_1,\hdots,x_n\}$.
 Note that the function $\dist(\cdot,P_n)$
 is contained in $F_d$ and vanishes on $P_n$.
 On the other hand, every other function $f\in F_d$
 that vanishes on $P_n$ must satisfy
 $$
  \abs{f(x)} \leq \dist(x,P_n)
  \qquad\text{for all } x\in [0,1]^d.
 $$
 This yields
 $$
  r(N_n)
  = \sup_{\substack{f\in F_d\\ N_n(f)=0}}\,\norm{f}_q
  = \norm{\dist(\cdot,P_n)}_q.
 $$
 Let us first consider the case $q=\infty$.
 Since the volume of the union of the balls 
 $B_r^\infty(x)$ over $x\in P_n$ is smaller than 1
% $$
%  \lambda^d\big(\bigcup_{y\in P_n} B_r(y)\big) < 1
% $$
 for all $r<1/(2m)$, there must be some $x\in [0,1]^d$
 with $\dist(x,P_n)\geq r$. Therefore,
 $$
  \norm{\dist(\cdot,P_n)}_\infty \geq \frac{1}{2m},
 $$
where equality holds when $P_n=\{i/m \,:\, 0\leq i<m\}^d$.
% $$
%  r(N_n)^q=
%  \sup_{\substack{f\in F_d\\ N_n(f)=0}}\,\int_{[0,1]^d} 
%  \abs{f(x)}^q~\dx
%  = \int_{[0,1]^d} \dist(x,P_n)^q~\dx.
% $$
 Let us turn to the case $q<\infty$.
Let $\lambda^d$ denote Lebesgue measure on $[0,1]^d$. Then
 $$
  r(N_n)^q
  = \int_{[0,1]^d} \dist(x,P_n)^q~\dint x
  = \int_0^{\infty} \lambda^d\brackets{\dist(x,P_n)^q\geq t}\dint t.
 $$
 Note that
 $$
  \lambda^d\brackets{\dist(x,P_n)^q\geq t}
  = 1-\lambda^d\brackets{\dist(x,P_n)^q< t}
  \geq 1-n\cdot 2^d t^{d/q},
 $$
 where equality holds if the sets
 $B_{t^{1/q}}^\infty(y)$ for $y\in P_n$
 are pairwise disjoint.
 Hence,
 \begin{align*}
  r(N_n)^q
  & \geq \int_0^{(1/2m)^q} \lambda^d\brackets{\dist(x,P_n)^q\geq t}\dint t \cr
  & \geq \frac{1}{(2m)^q} - 2^d n \int_0^{(1/2m)^q} t^{d/q}\dint t
  = \frac{1}{(2m)^q} \frac{d}{d+q}
 \end{align*}
 with equality for $P_n=\{i/m \,:\, 0\leq i<m\}^d$.
 This proves the statement.
\end{proof}

In the following, 
we study the expected radius of random information
$$
 N_n(f)=\big(f(x_1),\hdots,f(x_n)\big),
$$
where the points $x_1,\hdots,x_n$ are 
independent and uniformly distributed
in $[0,1]^d$.
If $q$ is finite, the $q^{\rm th}$ moment of the radius
at zero can be computed precisely.

\begin{prop}
\label{prop:lipnoninfty}
Consider $L_q$-approximation of functions from $F_d$ defined in 
\eqref{eq:defFdLip} 
with $1\le q < \infty$, 
using function values at uniformly distributed points in $[0,1]^d$.
%Let $1\leq q<\infty$ and $n\in\N$.
 Then
 $$
  \E \left[r(N_n)^q\right]
  = \frac{1}{2^q} \frac{n!}{(q/d+1)\cdots(q/d+n)}.
 $$
 In particular, we have
 $$
  \lim_{n\to\infty} n^{1/d} \sqrt[\leftroot{1}\uproot{2}q]{\E \left[r(N_n)^q\right]}
  \,=\,
  \frac{1}{2} \sqrt[\leftroot{1}\uproot{2}q]{\Gamma(q/d+1)}.
 $$
\end{prop}

\begin{proof}
 Let $P_n=\{x_1,\hdots,x_n\}$.
 Recall that
 $$
  r(N_n)^q=
  \int\nolimits_{[0,1]^d} \dist(x,P_n)^q~\dint x.
 $$
 Using Tonelli's theorem, we obtain
 $$
  \E \left[r(N_n)^q\right]
  = \int\nolimits_{[0,1]^d} \E\big[ \dist(x,P_n)^q\big]~\dint x.
 $$
 We will show that the integrand of the latter
 integral is constant.
 Let us fix $x\in[0,1]^d$ and
 note that $\dist(x,P_n)\in[0,1/2]$.
 For any $t\in[0,1/2]$, we have 
 $$
  \dist(x,P_n) \geq t
  \quad \Longleftrightarrow \quad
  \forall i\in\set{1,\hdots,n}:
  x_i \not\in B_t^\infty(x).
 $$
 Thus
 $$
  \PP[\dist(x,P_n) \geq t]
  = \brackets{1-(2t)^d}^n.
 $$
 The substitution $s=1-(2t^{1/q})^d$ and
 integration by parts yield
 \begin{multline*}
  \E\big[\dist(x,P_n)^q\big]
  =\int_0^{2^{-q}} \PP[\dist(x,P_n)^q \geq t]\,\dint t
  =\int_0^{2^{-q}} \brackets{1-(2t^{1/q})^d}^n\,\dint t\\
  =\frac{q/d}{2^q}\int_0^1 s^n (1-s)^{q/d-1}\,\dint s
%  =\frac{n}{2^q} \int_0^1 s^{n-1} (1-s)^{q/d}~\dint s
  =\frac{1}{2^q} \frac{n!}{(q/d+1)\cdots(q/d+n)}
 ,\end{multline*}
 which implies the statement of the theorem.
% We will show that
% \begin{equation}
%  \label{eq:expected dist}
%  \EE\brackets{\dist(x,P_n)}^q
%  = \frac{1}{2^q} \frac{n!}{(q/d+1)\cdots(q/d+n)}.
% \end{equation}
% This proves the theorem.
\end{proof}

% Since the expected radius is bounded
% above by its $q^{\rm th}$ moment
% and bounded below by the radius of optimal information,
% this leads to the following corollary.

% \begin{cor}
% \label{cor:Lip p<infty}
%  Let $1\leq q <\infty$. Then
%  $$
%   \EE (r(N_n))
%   \asymp r(N_n^*) 
%   \asymp n^{-1/d}.
%  $$
% \end{cor}

We now turn to the case of uniform approximation, i.e., $q=\infty$.
In this case,
the expected radius of information 
is closely related to 
the so called coupon collector's problem.
This is the question for
the random number $\tau_\ell$ of coupons that a coupon collector has
to collect to obtain a complete set of $\ell$ distinct coupons.
The following facts on the distribution of $\tau_\ell$ 
are well known, see \cite{LPW09}.
Here $H_\ell=\sum_{k=1}^\ell 1/k$ is the $\ell^{th}$
harmonic number. Note that $H_\ell\sim \log\ell$
as $\ell\to\infty$. % tends to infinity.

\begin{lemma}
\label{lem:coupons}
 Let $(Y_i)_{i=1}^\infty$ be a sequence of random
 variables that are uniformly distributed in
 the set $\set{1,\hdots,\ell}$
 and let
 $$
  \tau_\ell=\min\set{n\in\N \,:\, \set{Y_1,\hdots,Y_n}=\set{1,\hdots,\ell}}.
 $$
 Then
 $$
  \E[\tau_\ell]= \ell H_\ell
  \quad\text{and}\quad
  {\rm Var}[\tau_\ell]\leq \ell^2 \sum_{k=1}^\ell 1/k^2
 $$
 and for any $c\in(0,\infty)$,
 $$
  \PP\big[\tau_\ell > \lceil c\,\ell\log \ell\rceil\big] \leq 
  \ell^{-c+1}.
 $$
\end{lemma}

% \begin{proof}
%  For $1\leq i \leq \ell$, let $\nu_i$ be the
%  number of coupons that have to be collected
%  to get the $i$-th distinct coupon after having
%  collected $i-1$ distinct coupons.
%  These are independent geometric random variables with
%  $$
%   \EE\,\nu_i = \frac{\ell}{\ell-i+1}
%   \quad\text{and}\quad
%   \Var\nu_i= \frac{\ell(i-1)}{(\ell-i+1)^2}
%   \leq \frac{\ell^2}{(\ell-i+1)^2}.
%  $$
%  Now the first two statements follow from $\tau_\ell=\sum_{i=1}^\ell \nu_i$.
%  To obtain the tail bound, we consider the events
%  $A_i$ that the coupon with number $i$ was not
%  collected during the first $\lceil c\,\ell\log \ell\rceil$ trials.
%  Then
%  $$
%   \PP(A_i)=\brackets{1-1/\ell}^{\lceil c\,\ell\log\ell\rceil}
%   \leq \exp\brackets{-c\log \ell}
%   =\ell^{-c}.
%  $$
%  This yields
%  $$
%   \PP\brackets{\tau_\ell > \lceil c\,\ell\log \ell\rceil}
%   = \PP\brackets{\bigcup_{i=1}^\ell A_i}
%   \leq \sum_{i=1}^\ell \PP(A_i)
%   \leq \ell^{-c+1},
%  $$
%  as stated in the proposition.
% \end{proof}

This leads to the following estimates
of the expected radius for $q=\infty$.

\begin{prop} \label{thm:lipinfty}
Consider $L_\infty$-approximation of functions from $F_d$ defined in \eqref{eq:defFdLip}.
 For $n\in\N$, let
\begin{align*}
   m_1 &=\min\set{m\in\N \,:\, m^d (H_{m^d}-2) \geq n},\\
   m_2 &=\max\set{m\in\N \,:\, 2 m^d \log(m^d) \leq n}.
\end{align*}
 Then
 $$
  \frac{1}{4m_1}
  \leq \E[ r(N_n)] \leq
  \frac{2}{m_2}.
 $$
\end{prop}

\begin{proof}
Let $m\in\N$ and decompose $[0,1]^d$ into $\ell=m^d$ boxes 
 $$
  \prod_{i=1}^d\left[\frac{k_i-1}{m},\frac{k_i}{m}\right),
  \quad
  k_1,\dots,k_d\in\set{1,2,\hdots,m}
 $$
% $B_1,\hdots,B_\ell$ 
 of equal volume.
 Again, let $P_n=\{x_1,\hdots,x_n\}$.
 Recall that the radius of the information 
$N_n(f) = \big( f(x_1),\dots,f(x_n) \big)$ is given by
 $$
  r(N_n)
  = \max_{x\in [0,1]^d} \dist(x,P_n).
 $$ 
 Therefore, $r(N_n)$ is bounded above by $1/m$ if every box contains
 a point of $P_n$,
 and bounded below by $1/(2m)$ if one of the boxes does
 not contain a point of $P_n$.
 Let $A$ be the event that every box contains a point.
% with our random points $x_1,\hdots,x_n$.
 Note that the number of random points $x_i$
 that it takes to hit all the boxes follows the
 distribution of the coupon collector's variable $\tau_\ell$ 
 as defined in Lemma~\ref{lem:coupons}.
%  Thus
%  $$
%   \PP(A) = \PP\brackets{\tau_\ell \leq n}.
%  $$
 For the upper bound, we choose $m=m_2$.
 Lemma~\ref{lem:coupons} yields
 $$
  \PP\big[A^\mathsf{c}\big] = \PP[\tau_\ell > n] \leq 1/\ell
 $$
 and hence
 $$
  \E[ r(N_n)]\\
%  \leq \PP(A) \Xnorm{r(N_n)}{L_\infty(A)}
%  + \PP\brackets{A^\mathsf{c}} \Xnorm{r(N_n)}{L_\infty(A^\mathsf{c})}\\
 \leq \PP[A] \cdot \frac{1}{m} + \PP\big[A^\mathsf{c}\big] \cdot 1
 \leq \frac{2}{m}.
 $$
 For the lower bound, we choose $m=m_1$.
 Chebyshev's inequality yields
 $$
  \PP[A] = \PP[\tau_\ell\leq n]
  \leq \PP[\tau_\ell \leq \ell H_\ell-2\ell ]
  \leq \frac{{\rm Var}[\tau_\ell]}{4\ell^2}
  \leq \frac{1}{2}.
 $$
 We obtain
 $$
  \E[ r(N_n)]
  \geq \PP\big[A^\mathsf{c}\big]\, \frac{1}{2m} 
  %\min_{A^\mathsf{c}} r(N_n)
  \geq \frac{1}{4m},
 $$
 as it was to be proven.
\end{proof} 

We arrive at the following result 
on the power of random information
proved in~\cite{K19}.
Note that the case $q=\infty$ is already known from~\cite{BDKKW17},
where the authors study the uniform approximation
of functions on $[0,1]^d$ with bounded $r^{\rm th}$ derivative.

% The upper bound for Sobolev spaces on closed manifolds 
% can also be found in~\cite{EGO18}.

\begin{thm}[Krieg]
Consider $L_q$-approximation of functions from $F_d$ 
defined in \eqref{eq:defFdLip} with $1\le q \le \infty$, 
using function values at uniformly distributed points in $[0,1]^d$.
% Let $1\leq q \leq \infty$. 
Then 
 \[
 \E[r(N_n)] \asymp
 \begin{cases}
 r\brackets{N_n^*} \asymp n^{-1/d} & \text{if} \quad q<\infty \\
 & \\
 r\brackets{N_{n/\log n}^*} \asymp \brackets{\frac{n}{\log n}}^{-1/d} &\text{if} \quad q=\infty.
 \end{cases}
 \]

% $$
%  \EE\brackets{ r(N_n)}
%  \asymp r\brackets{N_n^*} 
%  \asymp n^{-1/d}
%  \qquad\text{for } q<\infty
% $$
% and
% $$
%  \EE\brackets{ r(N_n)}
%  \asymp r\brackets{N_{n/\log n}^*} 
%  \asymp \brackets{\frac{n}{\log n}}^{-1/d}
%  \qquad\text{for } q=\infty.
% $$
\end{thm}

\begin{proof}
 The statements on the rate of the optimal radius of information follow from Proposition~\ref{prop:lipoptimal}.
 The upper bound on the expected radius in the first case
 follows from Proposition~\ref{prop:lipnoninfty}, since the expected radius
 is bounded above by its $q^{\rm th}$ moment.
 The lower bound is trivial.
 The upper and lower bound on the expected radius in the second case
 follow from Proposition~\ref{thm:lipinfty}, since both $m_1^d$ and $m_2^d$ are of order $n/\log n$.
\end{proof}

Thus, in the sense of order of convergence,
random information is as good as optimal information
for the problem of $L_q$-approximation on $F_d$ if $q<\infty$.
If $q=\infty$, we loose a logarithmic factor.

\begin{rem}[Modifications of $F_d$]
 The rates of convergence of the expected and
 the optimal radius do not change
 if we replace the maximum metric on the torus
 by some equivalent metric.
 The same holds true if we change the
 Lipschitz constant or if we switch 
 to the non-periodic setting.
\end{rem}

\subsection{Classical Sobolev spaces}
\label{s2.3} 

We study the general $L_q$-approximation 
problem for Sobolev spaces, 
$$
\APP  \colon  W_p^s ([0,1]^d) \to L_q ([0,1]^d)
$$
given by $\APP(f)=f$ with standard information.
%One even may want to replace the domain $[0,1]^d$ 
%by an arbitrary bounded Lipschitz domain in $\R^d$.
A typical choice for the norm in this space is
\[
\|f\|_{W^s_p([0,1]^d)} \,=\,
\bigg(\sum_{\beta\in\N_0^d\colon \|\beta\|_1\le s } \|D^\beta f\|_p^p\bigg)^{1/p}.
\]
It is known that the radius of optimal information $N_n^*$ satisfies
\begin{equation}  \label{triebel}
r(N_n^\ast) \asymp n^{-s/d + (1/p-1/q)_+},
\end{equation}
where we assume that $s > d/p$ holds, to ensure that 
the functions are continuous.
This can be found in \cite{NT06} for general Lipschitz domains 
and was known earlier for 
special cases, such as the cube, see~\cite[Remark~24]{NT06}. 
Here we ask for the expected radius
of random information $N_n$ given by $n$ independent
and uniformly distributed 
points in $[0,1]^d$.
We conjecture that random information 
is as good as optimal information provided that $p>q$, 
whereas in the case $p\le q$ there is a loss of a logarithmic factor.
This conjecture is true in
the following special cases:
\begin{itemize}
 \item Section \ref{s2.1} covers univariate Sobolev
 spaces of smoothness $1$, i.e., $s=d=1$.
 \item Section \ref{s2.2} covers Lipschitz functions, 
 i.e., $s=1$ and $p=\infty$.
 \item The paper \cite{BDKKW17} covers uniform approximation
 on H\"older classes, i.e., $p=q=\infty$.
\end{itemize}
%\vspace*{-4mm}
In the general case, we are only able to prove the 
second part of our conjecture.

\begin{thm}
 \label{thm:Sobolev_small_q}
Consider $L_q$-approximation of functions from $W^s_p([0,1]^d)$ 
with $1\le p\le q \le \infty$, 
using function values at uniformly distributed points in $[0,1]^d$.
% Let $1\leq p \leq q \leq \infty$. 
Then
 $$
  \E[ r(N_n)]
  \asymp r\brackets{N_{n/\log n}^*} 
  \asymp \brackets{\frac{n}{\log n}}^{-s/d + 1/p - 1/q}.
 $$
\end{thm}

We use the following
criterion for optimal point sets, which can be found in \cite{NT06}.
%\er{It holds for the spaces 
%$W^s_p([0,1]^d)$ and many other isotropic spaces but not for the 
%mixed spaces that are discussed in Remark~\ref{rem3}.}
\begin{lemma}
\label{lem:optimal_point_sets}
Consider $L_q$-approximation of functions from $W^s_p([0,1]^d)$ 
with $1\le p, q \le \infty$.
For all $n\in\N$, let
$P_n\subset [0,1]^d$ be a point set of cardinality $n$,
 $$
  d_n=\min_{x,y\in P_n} \norm{x-y}_\infty,
  \quad\text{and}\quad
  r_n= \max_{y\in [0,1]^d}  \min_{x\in P_n} \norm{x-y}_\infty.
%   \min\big\{ r>0 \,\big\vert\, 
%   \forall y\in [0,1]^d \colon \exists x\in P_n \colon \d(x,y)\leq r
% %   \bigcup_{x\in P_n} B_r(x) \supseteq [0,1]^d 
%   \big\}.
 $$
 If $r_n \preccurlyeq d_n$,
 then the information $N_n$ with nodes $P_n$ satisfies 
$r(N_n) \asymp r(N_n^*)$.
\end{lemma}

\begin{proof}[Proof of Theorem~\ref{thm:Sobolev_small_q}]
 Let $\alpha=s/d - 1/p + 1/q$.
 We start with the lower bound.
 Let $g:\R^d\to \R$ be a smooth function with support in the 
$\ell_\infty$-unit ball that is not identically zero.
 Using Lemma~\ref{lem:coupons},
 we obtain with probability $1/2$ that $[0,1]^d$ contains an 
$\ell_\infty$-ball of radius  
$$
  \tau_n \asymp \brackets{\frac{\log n}{n}}^{1/d}.
$$  
not containing any of the points in $P_n$. If $x_0$ denotes the center of 
this ball, then the function $f\in W_p^s([0,1]^d)$ 
given by
 $f(x)=g(\tau_n^{-1}(x-x_0))$
 has support in this ball and satisfies
 $$
  \norm{f}_{W_p^s([0,1]^d)} \asymp \tau_n^{-s+d/p}
  \quad\text{and}\quad
  \norm{f}_q \asymp \tau_n^{d/q}.
 $$
 This yields
 $$
  \E[r(N_n)] \geq \frac{1}{2}\, \frac{\norm{f}_q}{\norm{f}_{W_p^s([0,1]^d)}}
  \asymp \tau_n^{s-d/p+d/q}
  \asymp \brackets{\frac{n}{\log n}}^{-\alpha}.
 $$
 We turn to the upper bound.
 We choose $m\in\N_0$ maximal such that $\ell=(3m)^d$ satisfies
 $$
  n\geq (\alpha +1) \ell \log \ell.
 $$
 We split the unit cube into $k=m^d$ subcubes of equal volume.
 Moreover, we split each of these into $3^d$ subcubes of equal volume.
 This gives us $\ell=(3m)^d$ small cubes.
 Let $A$ be the event that we hit every small cube.
 Lemma~\ref{lem:coupons} yields
 \begin{equation}
 \label{eq:small_pobability}
  \PP\big[A^\mathsf{c}\big] = \PP[\tau_\ell > n] 
 \leq \ell^{-\alpha}
 \asymp \brackets{\frac{n}{\log n}}^{-\alpha}.
 \end{equation}
 Assume now that $A$ takes place.
 We choose exactly one point out of the small cube
 in the center of every large cube
 to obtain a subset $P_k$ of $P_n$
 with cardinality $k$. 
%  Without loss of generality, let $P_k=\{x_1,\hdots,x_k\}$.
 Then
 $$
  r_k \leq \frac{2}{3m} \leq d_k.
 $$
 Let $N_k^\prime$ be the information map with nodes  $P_k$.
 Lemma~\ref{lem:optimal_point_sets} yields
 \begin{equation}
 \label{eq:small_radius}
  r(N_n)\leq r(N_k^\prime) \preccurlyeq k^{-\alpha}
  \asymp \brackets{\frac{n}{\log n}}^{-\alpha}.
 \end{equation}
 Together, \eqref{eq:small_pobability} and \eqref{eq:small_radius}
 imply the statement.
\end{proof}

We note that the authors of \cite{EGO18} study 
the integration problem for 
very general Sobolev spaces on closed manifolds. 
One can obtain the optimal rate $r(N_n^*) \asymp n^{-s/d}$. 
The authors prove the upper bound 
$$
\E [ r(N_n) ]  \preccurlyeq n^{-s/d} \, \log(n)^{s/d},
$$
where the random points are chosen with respect to the probability measure 
that defines the integral to be approximated.
The given proof works with the covering radius of random points
which is an important characteristic for $q=\infty$. 
For the integration problem or $q=1$, we expect, however, that 
the additional log factor is needed only for $p=1$.

\begin{rem} \label{rem3}
Another important case are tensor products of univariate 
Sobolev spaces. 
These spaces are usually called Sobolev spaces with (dominating) 
mixed smoothness or in the periodic case sometimes 
Korobov spaces.
One of the main features is that the optimal (main) orders 
of convergence are independent of the dimension, only the logarithmic 
terms depend on the dimension. 
We refer to \cite{DUT18} for a survey.  
There exist upper bounds for the numbers 
$\E [ r(N_n) ] $, but these bounds are weaker than for isotropic 
Sobolev spaces. 
One reason is that  
Lemma~\ref{lem:optimal_point_sets}
does not hold for these spaces 
and the geometrical structure of good sampling points 
is not known. 
See \cite{Bach17,HNWW01,HOeU19,KSF16,NW08,NW10,NW12}
for upper bounds on $\E [ r(N_n) ]$ in various cases. 
\end{rem}

\begin{rem}
Recovery of functions of many variables usually suffers from the 
curse of dimensionality, even for smooth $C^\infty$ functions,
see~\cite{NW09}. 
To avoid the curse one can impose structural properties, 
such as the sparsity with respect to the Fourier coefficients
or another orthonormal system. 
Then one can find an approximation of the function 
with relatively few random function values, 
see~\cite{chkifa2018,cohenetal10}.
Another way to study structural properties is to define 
weighted norms as was first done in \cite{SW98}, 
see \cite{NW08,NW10,NW12} for a survey.  
\end{rem}

\section{Uniformly distributed linear information} 
% % % % % % % % % % % % % % % % % % % % % % % % % % % 

In this section we study \textit{linear information} consisting
of arbitrary continuous linear functionals.
We are going to present two recovery problems for which we 
compare random information with optimal information. The first is a classical 
result due to Kashin, Garnaev, and Gluskin \cite{Ka1977, GG1984} on the 
recovery of $\ell_1^m$-vectors in the Euclidean distance, while the second 
are novel results of the authors on the recovery of vectors in an 
$m$-dimensional ellipsoid in the Euclidean distance \cite{HKNPU19}.

% % % % % % % % % % % % % % % % % % % % % % % % % % %
\subsection{Recovery of \texorpdfstring{$\ell_1^m$}{l\_1\textasciicircum m}-vectors in the Euclidean norm} 
% % % % % % % % % % % % % % % % % % % % % % % % % % % 

We study the problem of recovery of vectors in the unit ball of $\ell_1^m$ with general linear information and error measured in the Euclidean norm. This is the approximation problem with $F=\ell_1^m$ and $G=\ell_2^m$, i.e., $S$ is the identity map $\ell_1^m \hookrightarrow \ell_2^m$. The minimal worst case error, which is the minimal radius of information achievable with general linear information, is closely connected to the notion of Gelfand numbers or Gelfand widths. 
Gelfand numbers  and their dual counterpart Kolmogorov numbers are fundamental 
concepts in classical and modern approximation and complexity theory.  

For $n\in\N$, the $n^{\rm th}$ Gelfand number of a linear bounded operator $S:F \to G$ between normed spaces $F$ and $G$ can be defined as
$$
 c_n(S) = \inf_{\substack{F_n\subseteq F\\ {\rm codim}(F_n) \le n}}\sup_{ \|f\|_F\le 1, f \in F_n} \|S(f)\|_G.
$$
If the subspace $F_n$ of $F$ is the kernel of a linear information map $N_n:S\to \R^n$, then
$$
 r(N_n,0) = \sup_{ \|f\|_F\le 1, f \in F_n} \|S(f)\|_G
$$
is the local radius of zero information. 
It is not hard to see that
$$
 r(N_n,0) \le r(N_n) \le 2 r(N_n,0).
$$
Since every subspace $F_n$ of codimension at most $n$ is the kernel of a linear information map $N_n:S\to \R^n$, we have the well-known relation 
$$
 c_n(S) \le r ( N_n^* ) \le 2 c_n(S)
$$
between the radius of the optimal information $N_n^\ast$ for the linear problem $S:F \to G$ and the Gelfand numbers $c_n(S)$.

The systematic study of the Gelfand numbers of identity maps  $\ell_p^m \hookrightarrow \ell_q^m$
has a long tradition and is the basis for the study of Gelfand numbers of embeddings of Sobolev spaces and, therefore, also for the complexity of $L_q$-approximation of functions from Sobolev spaces.
%and can be traced back as far as the work of Stechkin~\cite{St1954}. After contributions of Stesin \cite{Ste1975}, Ismagilov \cite{I1974}, 
We want to discuss here the case $p=1$ and $q=2$ which   was one of the difficult cases. Its solution is an example for the power of random information. While lower bounds had been obtained before, the breakthrough regarding the upper bounds was made using random approximation, a groundbreaking method having its origin in the work of Kashin \cite{Ka1977}. The problem was finally settled in the series of papers by Kashin \cite{Ka1974, Ka1977}, Gluskin \cite{G81,G84}, and Garnaev and Gluskin \cite{GG1984}. 
%\[
%d^n(\mathbb B_q^m,\ell_p^m) \asymp
%\begin{cases}
%\max\bigg\{m^{1/p-1/q},\min\Big\{ 1,\frac{m^{1/q^*}}{n^{1/2}}\Big\} \sqrt{\frac{m-n}{m}}  \bigg\}  & : 1\leq q < 2 \leq p \leq \infty \\
%\bigg(\min\Big\{1,\frac{m^{1/q^*}}{n^{1/2}}  \Big\}\bigg)^{\frac{1/q^*-1/p^*}{1/q^*-1/2}} & : 1\leq q < p \leq 2 \\
%\max\bigg\{ m^{1/p-1/q},\bigg(\sqrt{\frac{m-n}{m}}\,\bigg)^{\frac{1/q-1/p}{1/2-1/p}} \bigg\} & : 2\leq p < q \leq \infty\,,
%\end{cases}
%\]
%where $\asymp$ denotes equivalence up to positive constants depending only on $p$ and $q$, but not on $n$ and $m$. 
%While Gluskin had obtained the (sharp) lower bounds already in \cite{G81}, the breakthrough regarding the upper bounds was made using random approximation, a method having its origin in the work of Kashin \cite{Ka1977}.

Let us consider now the particular case of the identity $\ell_1^m \hookrightarrow \ell_2^m$, which is of importance in information based complexity since it also demonstrates 
that non-linear recovery algorithms can be much better than linear ones. 
Although the results in the literature are usually formulated in the language of Gelfand numbers, we discuss the results directly for the radius of information. 

Recall that we now allow arbitrary linear information $N_n : \ell_1^m \to \R^n$. The random information we consider is Gaussian information, where $N_n=N_{n,m}$ is given by an $n \times m$-matrix with independent standard normal entries. The kernels of these maps are distributed according to the Haar measure on the Grassmannian manifold of $n$-codimensional subspaces in $\R^m$. In this sense, we can speak about uniformly distributed linear information.

In \cite[Theorem 1]{Ka1977}, Kashin introduced his fundamentally new approach 
via random subspaces and obtained that
\[
 \E\big[ r(N_{n,m}) \big] \preccurlyeq \frac{1}{\sqrt{n}}\bigg(1+\log\frac{m}{n}\bigg)^{3/2}.
\]
Through a refinement of Kashin's arguments, Garnaev and Gluskin \cite{GG1984} 
were able to decrease the power $3/2$ to a square root, while at the same time providing the matching lower bounds. 
Altogether, we have the following result. 

\begin{thm}[Kashin, Garnaev, Gluskin]\label{thm:KGG gelfand}
Consider the recovery problem 
of vectors from $\ell_1^m$ in the Euclidean norm and Gaussian information. Then
$$
  \E\big[ r(N_{n,m}) \big] \asymp r(N_{n,m}^\ast) \asymp 
  \min\bigg\{1,\sqrt{\frac{\log(1+\frac{m}{n})}{n}} \,\bigg\}.
$$
\end{thm} 
%As was explained in the discussion at the beginning of this section, 
%the order of the $n$-th Gelfand width corresponds to the optimal error bound for the recovery of $\ell_1^m$-vectors 
%in the Euclidean norm with $n$ pieces of linear information under the class $\Lambda^{\text{all}}$, i.e.,
%\[
%e\big(n,\mathbb B_1^m, \ell_2^m,\Lambda^{\text{all}}\big) \asymp \min\bigg\{1,\sqrt{\frac{\log(1+\frac{m}{n})}{n}} \,\bigg\},
%\]
%where $\asymp$ denotes equivalence up to positive absolute constants. In terms of the information complexity, the result shows that for (non-linear) algorithms based on linear information, we have 
%\[
%n(\varepsilon,m) \leq C \, \frac{\log m}{\varepsilon^2},
%\]
%with some absolute constant $C\in(0,\infty)$.

We want to sketch the proof of the upper bound 
$$
  \E\big[ r(N_{n,m}) \big] \preccurlyeq
  \min\bigg\{1,\sqrt{\frac{\log(1+\frac{m}{n})}{n}} \,\bigg\}
$$
of Theorem \ref{thm:KGG gelfand} and demonstrate the power of the probabilistic method of Kashin, Garnaev, and Gluskin. 

\noindent\textbf{Step 1.} As already explained, one considers a Gaussian random matrix $G=(g_{ij})_{i,j=1}^{n,m}$ with independent and identically distributed standard normal entries and defines the random information mapping 
\[
N_n:\R^m \to \R^n,\quad N_n(x)=Gx.
\]

\noindent\textbf{Step 2.} One defines the mapping $\phi$ to be
\[
\phi:\R^n \to\R^m, \quad y\mapsto \argmin\limits_{{x\in\R^m,\, N_n(x)=y}} \|x\|_1
\]
and considers the corresponding random algorithm $A_n = \phi\circ N_n$.

\noindent\textbf{Step 3.} One can now prove that with high probability the error of the random algorithm is small for any $x\in\mathbb B_1^m$. More precisely, if $\delta\in(0,\infty)$, then taking 
\[
n \asymp \frac{\log m}{\varepsilon^2} + \log \frac{2}{\delta}
\]
yields a worst case error 
\[
\sup_{x\in \mathbb B_1^m}\|x-A_n(x)\|_2 \leq \varepsilon
\]
for the random algorithm $A_n$ with probability at least $1-\delta$.

\medskip

As a matter of fact, one actually shows that most linear information mappings $N_n:\R^m \to \R^n$ lead to a small error and so random information is essentially as good as optimal information. 
We want to emphasize that no explicit construction for such a mapping is known. So the optimal information is not accessible in this case and one has to settle for random information. We also want to mention that linear algorithms corresponding to linear widths or approximation numbers
are much worse than nonlinear algorithms for this recovery problem.

\medskip

Let us mention that Garnaev and Gluskin actually obtained sharp bounds in the more general setting of recovery of vectors from $\ell_1^m$ in the $\ell_q^m$-norm or Gelfand numbers of the identity $\ell_1^m \hookrightarrow \ell_q^m$. If the random information $N_{n,m}$ is again Gaussian information and $1<q\le 2$, then
$$
  \E\big[ r(N_{n,m}) \big] \asymp r(N_{n,m}^\ast) \asymp 
  \min\bigg\{1,\left(\frac{\log(1+\frac{m}{n})}{n}\right)^{1-1/q} \,\bigg\}.
$$

Later, Gelfand widths have also attracted quite some attention in the area of 
compressive sensing over the last two decades, see, e.g., \cite{CRT2006,CGLP2012,CDdV2009,D2006}.
The goal of compressive sensing is to recover compressible vectors $x\in\R^m$ (those which are close to sparse vectors with only few non-zero coordinates) from $n$ pieces of incomplete linear information.
In fact, vectors in the unit ball of $\ell_p^m$ with $0< p\leq 1$ serve as a good model for sparse vectors. In this context, also the recovery of vectors in $\ell_p^m$ with $0< p\leq 1$ or the Gelfand numbers of the identity $\ell_p^m \hookrightarrow \ell_q^m$
were studied. Donoho \cite{D2006} and Foucart, Pajor, Rauhut, and T. Ullrich \cite[Theorem 1.1]{FPRU2010} extended the Garnaev-Gluskin result also to the case $0<p\le 1$ and $p<q\le 2$ and proved
$$
  \E\big[ r(N_{n,m}) \big] \asymp r(N_{n,m}^\ast) \asymp 
  \min\bigg\{1,\left(\frac{\log(1+\frac{m}{n})}{n}\right)^{1/p-1/q} \,\bigg\}.
$$
Again, it is not known how two explicitly construct  measurement maps 
performing almost as good as $N_{n,m}^\ast$. 

A further problem in compressed sensing is the matrix
recovery problem. The analogon to the recovery problem
of vectors from $\ell_p^m$ in the $\ell_q^m$-norm is the
recovery problem of matrices from Schatten $p$-classes
in the Schatten $q$-norm. The particular case $q=2$ is
is the Frobenius norm. This problem was studied in \cite{CDK2015,HPV2017}.

% % % % % % % % % % % % % % % % % % % % % % % % % % % % % % % % % % % % %
\subsection{Recovery of vectors from an ellipsoid} 
% % % % % % % % % % % % % % % % % % % % % % % % % % % % % % % % % % % % %

We continue with another recovery problem and recent results of the authors 
concerning the quality of random information for $\ell_2$-approximation, 
see \cite{HKNPU19}.

Let $F$ be a centered ellipsoid in $\R^m$ with semi-axes $\sigma_1\geq \dots \geq \sigma_m$. 
Clearly, $F$ can be seen as the unit ball of some Hilbert space. 
We want to recover
$x\in F$ from the data $N_n( x )\in \R^n$ with linear information mapping $N_n \in \R^{n \times m}$ and measure the error in the Euclidean distance. While the power of the information mapping is 
$$
  r(N_{n})
 = \inf_{\phi:\R^n\to \R^m} 
 \sup_{x\in F} \norm{\phi(N_{n}(x))-x}_2,
$$
it is known that, for linear problems in Hilbert spaces, the worst data is the zero data and so 
$$
  r(N_{n})
 = 
 \sup_{x \in F \cap E_n} \norm{x}_2,
$$
where $E_n$ is the kernel of $N_n$, see \cite{NW08,TWW88}.
In this case, the power or radius of optimal information is given 
by $r(N_n^*) = \sigma_{n+1}$.
 The question is now how good random information $N_n$ is for the $\ell_2$-recovery problem, when random information is provided by a Gaussian random matrix $G=(g_{ij})_{i,j=1}^{n,m}$? Is it comparable to optimal information? 

As is explained in \cite{HKNPU19}, rephrased in geometric terms, one is interested in the circumradius of the intersection of a centered ellipsoid $F$ in $\R^m$ with a random subspace $E_n$ of codimension $n$. While the maximal radius is the length of the largest semi-axis $\sigma_1$, the minimal radius is the length 
of the $(n+1)^{\rm st}$ largest semi-axis $\sigma_{n+1}$. But how large is the radius of a typical intersection?
Is it comparable to the minimal or the
maximal radius or does it behave 
completely different?

Using various probabilistic tools, such as exponential estimates for sums of chi-squared random variables, Gordon's min-max theorem for Gaussian processes, or estimates for the extreme singular values of (structured) Gaussian matrices, the authors obtained upper and lower bounds on the radius of Gaussian information for the $\ell_2$-approximation problem which hold with overwhelming probability. In many cases these bounds are sharp up to absolute constants and show that random information is comparable to optimal information. 

Let us only present one case which is of particular importance, where the semi-axes behave like the singular values of Sobolev embeddings \cite[Corollary]{HKNPU19}.

\begin{thm}[Hinrichs, Krieg, Novak, Prochno, Ullrich]
For $k\in\N$ let 
\[
  \sigma_k \asymp k^{-\alpha} \ln^{-\beta}(k+1) , 
\]
where $\alpha \in(0,\infty)$ and $\beta\in\R$. Then
\[
  \E[r(N_n)] \asymp
  \left\{\begin{array}{cll}
  		\sigma_1
        &
        \text{for} \ n< c_m,
        &
        \text{if} \ \alpha<1/2 \text{\, or \,} \beta\leq\alpha=1/2,
        \vspace*{2mm}
        \\
        \sigma_{n+1}\sqrt{\ln(n+1)}
        \quad
        &
        \text{for} \ n<\sqrt{m},
        &
        \text{if} \ \beta>\alpha=1/2,
        \vspace*{2mm}
        \\
        \sigma_{n+1}
        &
        \text{for} \ n<m,
        &
        \text{if} \ \alpha>1/2,
        \end{array}\right.
 \]
with
\medskip
 %where we can choose 
\[
 c_m=
 \left\{\begin{array}{cl}
        c\,m^{1-2\alpha} \ln^{-\max\{2\beta,0\}}m
        &
        \text{for} \ \alpha<1/2,
        \\
        c \ln^{1-\max\{2\beta,0\}} m
        &
        \text{for} \ \beta<\alpha=1/2,
        \\
        c \ln\ln m
        &
        \text{for} \ \beta=\alpha=1/2,
        \end{array}\right.
 \]
where $c\in(0,\infty)$ is an absolute constant.
\end{thm}
This means that random information is 
just as good as optimal information
if the singular values decay with
a polynomial rate greater than $1/2$.
Then, in geometric terms, the size  of a typical intersection ellipsoid 
is comparable to the size of the smallest 
intersection.
On the other hand, if the singular values decay too slowly,
random information is rather useless, which means that 
a typical intersection ellipsoid
is almost as large as the largest.

The ellipsoid recovery problem can be considered in an infinite dimensional setting as well, even though our geometric intuition might fail. In this case, where $m=\infty$, we consider the ellipsoid
$$
 F = \Bigg\{ x \in \ell_2 \colon \sum_{j\in\N} \frac{x_j^2}{\sigma_j^2} \leq 1 \Bigg\}
$$
and the matrix $(g_{ij})_{1\leq i \leq n, j\in \N}$ with independent standard 
Gaussian entries as random information mapping $N_n$.

As a consequence
 of the probabilistic estimates derived in \cite[Theorems 3 and 4]{HKNPU19}, 
one obtains an $\ell_2$-dichotomy showing that random information is useful if and only if $\sigma\in\ell_2$ \cite[Corollary 5]{HKNPU19}.

\begin{thm}[Hinrichs, Krieg, Novak, Prochno, Ullrich]
 \label{cor:l2 not l2}
 If $\sigma\not\in\ell_2$, then $r(N_n)=\sigma_1$
  holds almost surely for all $n\in\N$.
 If $\sigma\in\ell_2$, then
 $$
  \lim_{n\to\infty}\sqrt{n}\,\E[ r(N_n)] = 0.
 $$
\end{thm}

\begin{rem}
Such an $\ell_2$-dichotomy is known from a related problem. There $F$ is the unit ball of a reproducing kernel Hilbert space $H$, 
i.e., $H\subset L_2(D)$ consists of functions 
on a common domain $D$ and
the evaluation functional $f \mapsto f(x)$ is 
a bounded operator on $H$ for every $x\in D$. 
The optimal linear information $N_n$ for the $L_2$-approximation problem is given by the singular value decomposition and has radius $\sigma_{n+1}$. 
This information might be difficult to implement and hence one 
might allow only information $N_n$ of the form 
\[
N_n (f) = \big( f(x_1), \dots , f(x_n)\big)\,, \quad x_1,\dots,x_n\in D.
\]
The goal is to relate the power of 
function evaluations
to the power of all continuous 
linear functionals and one would like to
prove that their power is roughly the same.
Unfortunately,
this is \textit{not} true in general.
When $\sigma \notin \ell_2$
the convergence of optimal algorithms 
that may only use function values
can be arbitrarily slow 
\cite{HNV08}.
The situation is much better if we assume that $\sigma \in \ell_2$.
It was shown in \cite{KWW09} and \cite{WW01} 
that function values are almost as good as general 
linear information. 
We refer to \cite[Chapter 26]{NW12} for a presentation of 
these results. 
\end{rem}

\section*{Acknowledgement}
The authors would like to thank the Isaac Newton Institute for 
Mathematical Sciences for support and hospitality during the programme
`Approximation, sampling and compression in data science' 
when work on this paper was undertaken. 
This work was supported by EPSRC Grant Number EP/R014604/1.

A. Hinrichs and J. Prochno are supported by the Austrian Science Fund (FWF) 
Projects F5509-N26 and F5508-N26, which are part of the Special Research 
Program ``Quasi-Monte Carlo Methods: Theory and Applications''. 
A. Hinrichs was also supported by a grant from the
Simons Foundation.
J. Prochno is also supported by a Visiting Professor Fellowship 
of the Ruhr University Bochum and its Research School PLUS.

\bibliographystyle{plain}
%\bibliography{HKNPU19}

\end{document}